\documentclass[11pt,a4paper]{article}
\usepackage{amssymb} 
\usepackage{amsmath} 
\usepackage{amsthm} 
\usepackage{graphicx}

\newtheorem{theorem}{Theorem}
\newtheorem{proposition}[theorem]{Proposition}
\newtheorem{lemma}[theorem]{Lemma}

\renewcommand{\Pr}{\,\mathbb{P}}
\newcommand{\eps}{\varepsilon}

\DeclareMathOperator{\ch}{ch}

\title{List colouring with a bounded palette}

\author{
Marthe Bonamy\thanks{This research was carried out during a visit by this author to Utrecht University. This author is supported by the ANR Grant EGOS (2012-2015) 12 JS02 002 01.} \\
Universit\'e Montpellier 2 -- LIRMM \\
{\tt marthe.bonamy@lirmm.fr}
\and
Ross J. Kang\thanks{This author was supported by a NWO Veni grant (project number 639.031.138).} \\
Radboud University Nijmegen \\
{\tt ross.kang@gmail.com}
}

\begin{document}

\maketitle

\begin{abstract}

Kr\'al' and Sgall (2005) introduced a refinement of list colouring where every colour list must be subset to one predetermined palette of colours. We call this {\em $(k,\ell)$-choosability} when the palette is of size at most $\ell$ and the lists must be of size at least $k$.
They showed that, for any integer $k\ge 2$, there is an integer $C=C(k,2k-1)$, satisfying $C = O(16^{k}\ln k)$ as $k\to \infty$, such that, if a graph is $(k,2k-1)$-choosable, then it is $C$-choosable, and asked if $C$ is required to be exponential in $k$.
We demonstrate it must satisfy $C = \Omega(4^k/\sqrt{k})$.

For an integer $\ell \ge 2k-1$, if $C(k,\ell)$ is the least integer such that a graph is $C(k,\ell)$-choosable if it is $(k,\ell)$-choosable, then we more generally supply a lower bound on $C(k,\ell)$, one that is super-polynomial in $k$ if $\ell = o(k^2/\ln k)$, by relation to an extremal set theoretic property. By the use of containers, we also give upper bounds on $C(k,\ell)$ that improve on earlier bounds if $\ell \ge 2.75 k$.
\end{abstract}


\section{Introduction}
\label{sec:intro}

The classic concept of {\em list colouring}, where an adversary may place individual restrictions on the colours used at each vertex of the graph, was introduced independently by Erd\H{o}s, Rubin and Taylor~\cite{ERT80} and Vizing~\cite{Viz76}.
We consider the ``bounded palette'' refinement of list colouring as defined by Kr\'al' and Sgall~\cite{KrSg05}.
Let $G = (V,E)$ be a simple, undirected graph. 
For any given positive integer $\ell$, we shall refer to $[\ell] = \{1,\dots,\ell\}$ as a {\em palette} (of colours).  Given a positive integer $k\le \ell$, a mapping $L: V\to \binom{[\ell]}{k}$ is called a {\em $(k,\ell)$-list-assignment} of $G$; a colouring $c$ of $V$ is called an {\em $L$-colouring} if $c(v)\in L(v)$ for any $v\in V$. We say $G$ is {\em $(k,\ell)$-choosable} if for any $(k,\ell)$-list-assignment $L$ of $G$ there is a proper $L$-colouring of $G$.  We say $G$ is {\em $k$-choosable} if it is $(k,\ell)$-choosable for any $\ell\ge k$.  The {\em choosability $\ch(G)$} (or {\em choice number} or {\em list chromatic number}) of $G$ is the least $k$ such that $G$ is $k$-choosable.
Note $G$ is properly $k$-colourable if and only if it is $(k,k)$-choosable.

A natural question one may wonder is whether $k$-choosability may be verified merely by establishing $(k,\ell)$-choosability with a large enough choice of $\ell$ as a function of $k$, independent of the given graph\footnote{Kierstead~\cite{Kie00} proved that $G =(V,E)$ is $k$-choosable if it is $(k,|V|)$-choosable.}. 
If true, this would immediately yield for fixed $k$ an algorithm for checking if a given input graph is not $k$-choosable that runs in time that is singly-exponential in  the number of vertices~\cite{Epp10}.
However, this question was answered by Kr\'al' and Sgall mainly in the negative. 
\begin{theorem}[\cite{KrSg05}]\label{thm:KrSg,negative}
For integers $k$ and $\ell$ satisfying $\ell\ge k\ge3$, there is a graph $G_{k,\ell}$ that is $(k,\ell)$-choosable but not $(k,\ell+1)$-choosable.
On the other hand, if a graph is $(2,4)$-choosable, then it is $2$-choosable.
\end{theorem}
\noindent
The graphs $G_{k,\ell}$ they construct are not too large, having $O(\ell^2)$ vertices. Their proof of Theorem~\ref{thm:KrSg,negative} used ideas of precolouring (non)extension.

Upon learning this, one might wonder if $(k,\ell)$-choosability of a graph at least provides partial evidence of choosability: does it imply the graph is $C$-choosable for some (possibly large) constant $C = C(k,\ell)$?  The positive answer to this second question is the content of the next result, also due to Kr\'al' and Sgall. This was later strengthened by the second author~\cite{Kan13} by a connection with Property~B (also known as weak $2$-colourability of uniform hypergraphs); 
a more precise version is reviewed in Theorem~\ref{thm:mindeg}.

\begin{theorem}[\cite{KrSg05}, cf.~\cite{Kan13}]\label{thm:KrSg,positive}
For integers $k$ and $\ell$ satisfying $k\ge 2$ and $\ell \ge 2k-1$, there is an integer $C = C(k,\ell)$ satisfying $C = O(16^{k}\ln k)$ 
such that, if a graph is $(k,\ell)$-choosable, then it is $C$-choosable.

Moreover, as $k\to\infty$, if for some fixed $b > 2$ we have $\ell \sim b k$, then $C$ may be chosen to satisfy
\begin{align*}
C \le (4(b-2)^{b-2}(b-1)^{2-2b}b^b +o(1))^k.
\end{align*}
\end{theorem}

\noindent
We remark that
\begin{align*}
\lim_{b\downarrow2}4(b-2)^{b-2}(b-1)^{2-2b}b^b = 16 \ \ \text{ and } \ \ \lim_{b\to\infty}4(b-2)^{b-2}(b-1)^{2-2b}b^b = 4.
\end{align*}
Observe that the condition $\ell \ge 2k-1$ cannot be ignored, because every bipartite graph is $(k,2k-2)$-choosable and the class of bipartite graphs has unbounded choosability.
Recently, Alon {\em et al.}~\cite{AKRWZ14+} sharpened the boundary between $\ell=2k-2$ and $\ell=2k-1$ by exhibiting high girth, bipartite, non-$(k,2k-1)$-choosable graphs all proper subgraphs of which have average degree at most $2k-2$.

Theorem~\ref{thm:KrSg,positive} builds upon the relationship between a graph's degeneracy and its choosability. It is easy to see by a greedy argument that, if every subgraph of a graph has a vertex of degree at most $d$, then the graph's choosability is at most $d+1$. Alon~\cite{Alo93,Alo00} showed with probabilistic methodology that a (weak) converse of this statement is also true.
In slightly more detail, Theorem~\ref{thm:KrSg,positive} is proved by a modification of the proof by Alon~\cite{Alo00} that any graph is $O(4^{k}k^4)$-degenerate if it is $(k,k^2)$-choosable.

Our first result uses the containers method to improve upon Theorem~\ref{thm:KrSg,positive}. This method was introduced recently by Saxton and Thomason~\cite{SaTh12,SaTh15}, who sought to more deeply understand the aforementioned relationship between degeneracy and choice number. As part of a broader approach to several important problems in random and extremal graph theory (cf.~also~\cite{BMS15}), they used this method to show that every graph $G$ with average degree $d$ satisfies $\ch(G) \ge (1+o(1))\log_2 d$ as $d\to\infty$, an asymptotically optimal statement. We follow this same approach for bounded palette choosability. Although this does not (yet) yield optimal results in our setting, it gives marked improvements in a large range of choices of $\ell$.

\begin{theorem}\label{thm:containers}
In Theorem~\ref{thm:KrSg,positive}, as $k\to\infty$, if for some fixed $b > 2$ we have $\ell \sim b k$, then $C$ may be chosen to satisfy
\begin{align*}
C \le (2(b-2)^{-1}b+o(1))^k.
\end{align*}
\end{theorem}

\noindent
Observe that $\lim_{b\downarrow2}2(b-2)^{-1}b = \infty$ and $\lim_{b\to\infty}2(b-2)^{-1}b = 2$.
Theorem~\ref{thm:containers} improves on Theorem~\ref{thm:KrSg,positive} when $((b-2)(b-1)^{-2}b)^{b-1} > 1/2$ which is roughly when $b$ is at least $2.747655083$.
See Figure~\ref{fig:comparison} for a comparison.

In the light of Theorems~\ref{thm:KrSg,negative} and~\ref{thm:KrSg,positive}, Kr\'al' and Sgall posed two natural follow-up questions.
\begin{enumerate}
\item For each $k$, what is the least $\ell^* = \ell^*(k)$, if it exists, such that every graph is $(k+1)$-choosable if it is $(k,\ell^*)$-choosable?
\item Must the smallest possible choice of $C(k,2k-1)$ in Theorem~\ref{thm:KrSg,positive} grow exponentially in $k$?
\end{enumerate}
Our second result answers the second of these questions in the affirmative and also provides a lower bound on the quantity $\ell^*$ in the first.
It also gives exponential lower bounds on the best possible choice of $C(k,\ell)$ in Theorem~\ref{thm:KrSg,positive} when $\ell = O(k)$.

\begin{theorem}\label{thm:bipartite}
For integers $k$ and $\ell$ satisfying $k\ge 2$ and $\ell \ge 2k-1$, there is a constant $R = R(k,\ell)$ satisfying $R \ge \exp((k-1)^2/\ell)$
such that the complete bipartite graph $K_{R-1,(R-1)^{R-1}}$ is $(k,\ell)$-choosable but not $R$-choosable.

Moreover, as $k\to\infty$, if for some fixed $b > 2$ we have $\ell \sim b k$, then $R$ may be chosen to satisfy
\begin{align*}
R \ge ((b-2)^{b-2}(b-1)^{2-2b}b^b +o(1))^k.
\end{align*}
\end{theorem}

\noindent
This improves upon and simplifies Theorem~\ref{thm:KrSg,negative} in certain cases, albeit with a larger graph.
Note that $R(k,\ell)$ is super-polynomial in $k$ if $\ell = o(k^2/\ln k)$, implying the hypothetical $\ell^*$ in the first question above must  be $\Omega(k^2/\ln k)$. 
We will see below that $R(k,2k-1)$ can be chosen as $\binom{2k-1}{k} \sim 4^k/(2\sqrt{\pi k})$ as $k\to\infty$, while $R(k,k^2) = k$.
The definition of $R(k,\ell)$ is based on what we call ``Property~K'', which is related to Property~B mentioned earlier; we do not know if it has been studied before.

\begin{figure}
\centering
\includegraphics[width=0.8\textwidth]{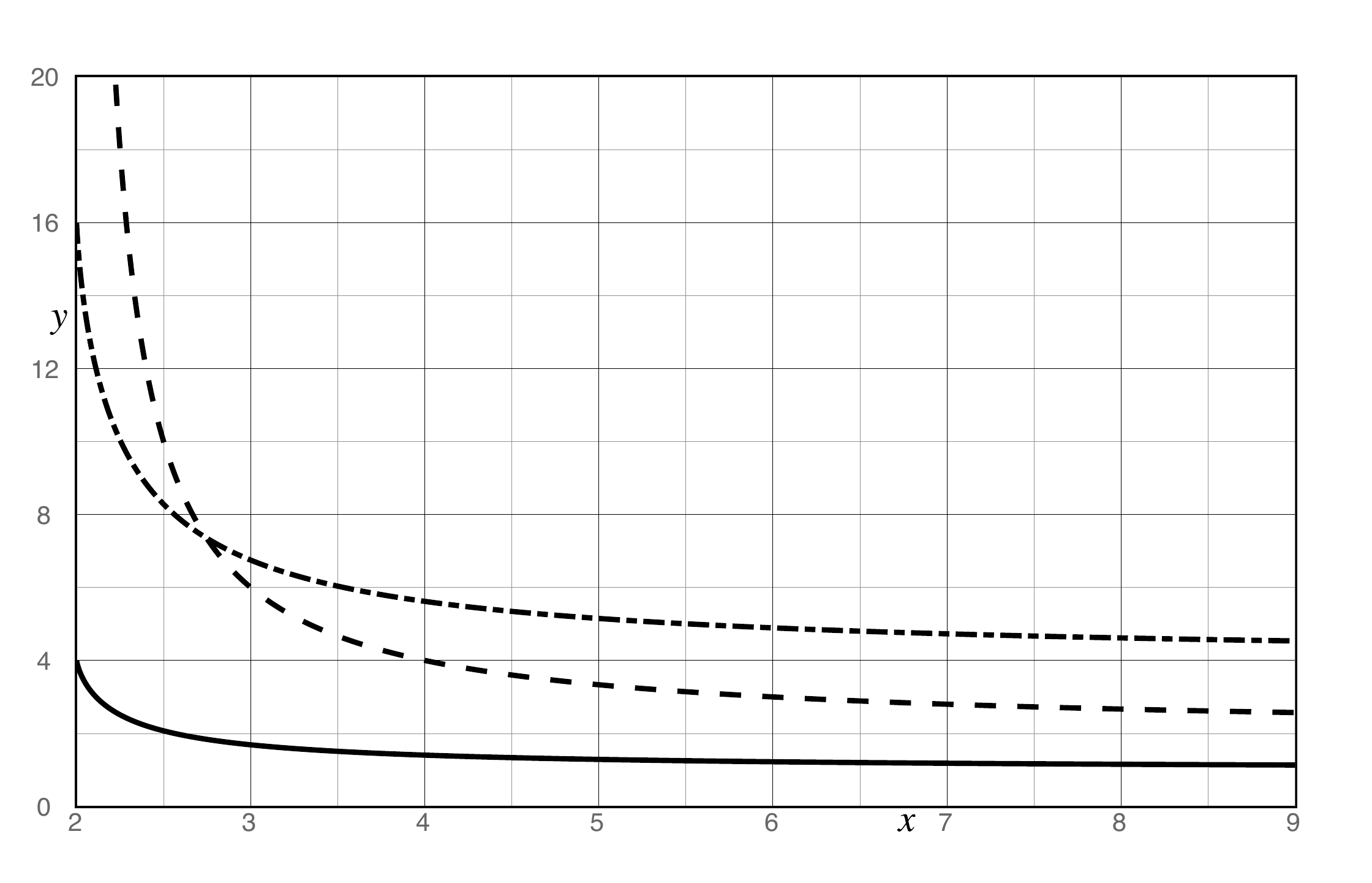}
\caption{A comparison plot. 
The curve $y = 4(x-2)^{x-2}(x-1)^{2-2x}x^x$ (Theorem~\ref{thm:KrSg,positive}) is dash--dotted;
$ y = 2(x-2)^{-1}x$ (Theorem~\ref{thm:containers}) is dashed;
and $y = (x-2)^{x-2}(x-1)^{2-2x}x^x$ (Theorem~\ref{thm:bipartite}) is solid.
\label{fig:comparison}}
\end{figure}

We remark that the following easy proposition settles the case $k=2$ for the first question of Kr\'al' and Sgall.  The proof of this is left to the reader.
\begin{proposition}
If a graph is $(2,3)$-choosable, then it is $3$-choosable.
\end{proposition}

The structure of the paper is as follows. In Section~\ref{sec:propB}, we  review the established connections between bounded palette list colouring and Property~B, and sketch the argument behind Theorem~\ref{thm:KrSg,positive}. In Section~\ref{sec:containers}, we indicate the immediate improvement upon Theorem~\ref{thm:KrSg,positive} available by use of the containers method, which establishes Theorem~\ref{thm:containers}. In Section~\ref{sec:propK}, we introduce Property~K and give a proof of Theorem~\ref{thm:bipartite}.

\section{Property~B and $(k,\ell)$-choosability}
\label{sec:propB}

Erd\H{o}s, Rubin and Taylor~\cite{ERT80} already noticed the close connection between choice number and the extremal study of Property~B~\cite{Erd63,Erd64,Erd69,ErHa61}.
In order to provide extra context and background, we here summarise this connection, especially with respect to list colouring with a bounded palette.

A family $\mathcal{F}$ of sets has {\em Property~B} if there exists a set $B$ which meets every set in $\mathcal{F}$ but contains no set in $\mathcal{F}$.
Property~B for a family of $k$-sets is equivalent to weak $2$-colourability of $k$-uniform hypergraphs.

For a fixed integer $k\ge2$, let $M(k)$ be the cardinality of a smallest family of $k$-sets that does not have Property~B.
For fixed integers $k,\ell\ge 2$, let $M(k,\ell)$ be the cardinality of a smallest family of $k$-subsets of $[\ell]$ that does not have Property~B.  Note that $M(k,2k-1)=\binom{2k-1}{k}$ since the collection $\binom{[2k-1]}{k}$ of all $k$-subsets of $[2k-1]$ does not have Property~B, whereas any proper subcollection of $\binom{[2k-1]}{k}$ has Property~B.  It also holds that $M(k,\ell) = \infty$ if $\ell \le 2k-2$, as every subcollection of $\binom{[2k-2]}{k}$ has Property~B.  Clearly, $M(k) = \inf_{\ell\ge 2k-1}M(k,\ell)$.

The best general upper bound on $M(k)$ is a probabilistic construction of Erd\H{o}s~\cite{Erd64} from the 1960's, while the best lower bound is a more recent application of the semirandom method by Radhakrishnan and Srinivasan~\cite{RaSr00} (a short proof of which was obtained recently by Cherkashin and Kozik~\cite{ChKo15}):
\begin{align}
\label{eqn:propB}
\Omega\left(2^k\sqrt{\frac{k}{\ln k}} \right)
\le
M(k)
\le
O\left(2^k k^2\right).
\end{align}
More tailored bounds on $M(k,\ell)$ were shown by Erd\H{o}s~\cite{Erd69}: there is some algebraic decreasing function $f : [2,\infty) \to \mathbb R$ satisfying $\lim_{b\downarrow2} f(b) = 4$ and $\lim_{b\to\infty} f(b) = 2$ such that,
if $\ell \ge 2k-1$ and $\ell \sim bk$ as $k\to\infty$, then $M(k,\ell) = (f(b)+o(1))^k$.
More fully,
\begin{align}\label{eqn:erdos}
f(b) = 2(b-2)^{\frac12(b-2)}(b-1)^{1-b}b^{\frac12b}
\end{align}

We next state the connections between the parameters $M(k,\ell)$ and $(k,\ell)$-choosability.  We first note the following easy proposition which can be derived quite naturally from the definition of $M(k,\ell)$.

\begin{proposition}[\cite{ERT80}]
\label{prop:bipartite,propB}
Let $k,\ell$ be integers such that $2\le k\le \ell$.
\begin{enumerate}
\item
If $n_1\ge M(k,\ell)$ and $n_2\ge M(k,\ell)$, then the complete bipartite graph $K_{n_1,n_2}$ is not $(k,\ell)$-choosable.
\item
Any bipartite graph with fewer than $M(k,\ell)$ vertices is $(k,\ell)$-choosable.
\end{enumerate}
\end{proposition}

The next result (given in a slightly more general form in~\cite{Kan13}) extends a $(k,2k-1)$-choosability version due to Kr\'al' and Sgall.

\begin{theorem}[\cite{KrSg05}, cf.~\cite{Kan13}]
\label{thm:mindeg}
Let $k,\ell$ be integers such that $2\le k\le \ell$ and
\begin{align*}
D = 12 (M(k,\ell))^2 \cdot \ln M(k,\ell) \cdot \ln k \cdot
\left(1+\sqrt{1+\frac{1}{3\ln M(k,\ell)}}\right)^2.
\end{align*}
Any graph with minimum degree $D$ is not $(k,\ell)$-choosable.
\end{theorem}

\noindent
Now, together with the fact that any $d$-degenerate graph is $(d+1)$-choosable,
the first part of Theorem~\ref{thm:KrSg,positive} holds by monotonicity of $M(k,\ell)$ in $\ell$ and the fact that $M(k,2k-1) = \binom{2k-1}{k} = O(4^k/\sqrt{k})$, while the second part is implied by the result of Erd\H{o}s associated to the expression of~\eqref{eqn:erdos}.

As mentioned in the introduction, the proof strategy for Theorem~\ref{thm:mindeg} is after Alon~\cite{Alo93,Alo00}.  It has two stages of randomness. In the first, we choose a small random vertex subset $A$ and assign lists independently and uniformly at random from $\mathcal F$ to the vertices of $A$, where ${\mathcal F} \subseteq\binom{[\ell]}{k}$ is some family not having Property~B. With positive probability, 
there must be a large number of ``good'' vertices, that is, vertices outside of $A$ having for every $F\in {\mathcal F}$ a neighbour in $A$ with list $F$. We fix some such $A$ and its list-assignment. 
In the second stage, we assign lists independently and uniformly at random from $\mathcal F$ to the good vertices, from which we can show that with positive probability no valid list colouring is possible. We refer to~\cite{Kan13} for the details.

\section{Containers and $(k,\ell)$-choosability}
\label{sec:containers}

In this section, we improve upon Theorem~\ref{thm:mindeg} and hence Theorem~\ref{thm:KrSg,positive} when $k$ is large and $\ell \ge 2.75 k$. To do so, we use the recently-introduced containers method. 
We require a more general containers theorem of Saxton and Thomason~\cite{SaTh15} and one of its specific consequences. We remark that, independently, Balogh, Morris and Samotij~\cite{BMS15} obtained a similar theorem with a similar (wide) array of important consequences, except that they did not target list colouring.
List colouring was the original motivation of Saxton and Thomason in formulating the concept of containers in~\cite{SaTh12}.
The following is an analogue of Theorem~2.1 in~\cite{SaTh15}, adapted for $(k,\ell)$-choosability of graphs and reformulated in our notation. 
It implies Theorem~\ref{thm:containers}.

\begin{theorem}\label{thm:averagedegree}
Let $b > 2$.
There is a function $d = d(k)$ satisfying as $k\to \infty$
\begin{align*}
d = (2(b-2)^{-1}b+o(1))^k
\end{align*}
such that any graph with average degree $d$ is not $(k,\lfloor bk\rfloor)$-choosable.
\end{theorem}

The idea of the containers method is that in order to get a reasonable understanding of the independent sets of a (hyper)graph (a task that frequently arises in probabilistic and extremal combinatorics), it often suffices to work with some good collection of container vertex subsets. By ``good'', we mean that for each independent set there is some container which has it as a subset, that the number of containers is small (and in particular much smaller than the number of independent sets), and that each container is not large. There is already a rather broad and useful collection of interpretations for ``small'' and ``not large'' for which the following statement holds: every (hyper)graph of average degree $d$ has a good collection of containers~\cite{BMS15,SaTh15}.

Let $G=([n],E)$ be a graph and suppose we generate a $(k,\ell)$-list-assignment $L$ of $G$ by assigning each list uniformly at random from $\binom{[\ell]}{k}$. If there exists a proper $L$-colouring $f$, then for each colour $i\in[\ell]$ the set of vertices $u$ with $f(u)=i$ is an independent set; in particular, there is a collection of independent sets $(I_1,\dots,I_\ell)$ such that $u\in I_{f(u)}$ for all $u\in [n]$.
For {\em any} collection of sets $(C_1,\dots,C_\ell)$, we say $L$ is {\em compatible} with $(C_1,\dots,C_\ell)$ if there is a function $f: [n]\to[\ell]$ such that $f(u)\in L(u)$ and $u\in C_{f(u)}$ for all $u\in [n]$.
If we can find a $(k,\ell)$-list-assignment $L$ that is incompatible with every $(I_1,\dots,I_\ell)\in {\mathcal I}^\ell$ where $\mathcal I$ is the collection of all independent sets, then it will follow that $\ch(G) > k$.
However, this conclusion also follows from finding some $L$ that is incompatible with every $(C_1,\dots,C_\ell) \in {\mathcal C}^\ell$ where $\mathcal C$ is a good collection of containers. That the collection $\mathcal C$ is ``small'' and each container is ``not large'' is essentially what is sufficient to prove the existence of a desired (incompatible) list-assignment $L$ by the probabilistic method.

We have superficially described the approach to proving Theorem~\ref{thm:averagedegree} and point to Section~8 of~\cite{SaTh15} for a better and fuller explanation of the details in the proof of their Theorem~2.1.
Those details are substantial, but there is one main point where we differ, namely, the following $(k,\lfloor bk\rfloor)$-choosability version of Lemma~8.1 in~\cite{SaTh15} --- they had $\ell = \Theta(k^2)$ instead of $\ell = O(k)$.

\begin{lemma}\label{lem:containers}
Let $0 < \eps, c < 1$ and $b > 1/c$. Then there exists $x_0 = x_0(\eps,c,b)$ such that the following holds for all $x> x_0$.

Let $k = \lfloor(1-\eps)\log x/\log(1/(c-1/b))\rfloor$ and let $\ell = \lfloor b k\rfloor$.
Let $n > x$ and let ${\mathcal C} \subseteq 2^{[n]}$.
Suppose that there is a map $g: {\mathcal C}^{\ell} \to [x,n]$ such that
\begin{align*}
\frac{1}{\ell} \sum_{i=1}^{\ell} |C_i\cap [v]| \le (1-c)v
\end{align*}
holds for each $(C_1,\dots,C_\ell)\in {\mathcal C}^\ell$ where $v = g(C_1,\dots,C_\ell)$. Suppose also that
\begin{align*}
|\{(C_1\cap[v],\dots,C_\ell\cap[v]) : g(C_1,\dots,C_\ell) = v\}| \le \exp(v\ell/x)
\end{align*}
holds for all $v\in [n]$. Then there is a $(k,\ell)$-list-assignment that is incompatible with every $(C_1,\dots,C_\ell) \in {\mathcal C}^\ell$.
\end{lemma}

\begin{proof}[Proof outline]
The proof closely follows that of Lemma~8.1 in~\cite{SaTh15} after the appropriate substitution of expressions for $k$ and $\ell$ (which are, respectively, $\ell$ and $t$, in their choice of letters).

The condition $b > 1/c$ ensures that $c\ell > k$ so that $(c-(k-1)/\ell)^k \ge (c-1/b)^k$ has a positive base.
The choice of $k$ ensures that $(c-1/b)^k \ge x^{\eps-1}$.
For more details we direct the reader to~\cite{SaTh15}.
\end{proof}

To complete the proof of Theorem~\ref{thm:averagedegree}, we use the same approach as for the $r=2$ case of Theorem~2.1 in~\cite{SaTh15}. Starting with a graph of average degree $d$, we apply a containers theorem, Theorem~3.7 in~\cite{SaTh15}. (It is necessary that Theorem~3.7 as stated in~\cite{SaTh15} is valid for all $t\in\mathbb N$.) Then we feed the resulting collection ${\mathcal C}$ of containers as input to Lemma~\ref{lem:containers}. Note that the choice of parameters will be such that $c = 1/2 + o(1)$ (for graphs), $\eps = o(1)$ and $\log x = (1+o(1))\log d$ as $d\to\infty$. The list-assignment given as output by Lemma~\ref{lem:containers} then certifies that the graph is not $(k,\lfloor bk\rfloor)$-choosable.  We omit the remaining details and refer the reader to~\cite{SaTh15}.

The containers method is powerful, and our goal in this section was only to indicate an immediate improvement with this method in our setting.
It is worth pointing out that the consequences for $C(k,\ell)$ are inferior to those of the previous section when $\ell$ is close to $2k-1$.
In particular, the method used to obtain Theorem~\ref{thm:averagedegree} is insufficient to show that for every $k\ge 2$ there is some $C$ such that any graph is $C$-choosable if it is $(k,2k-1)$-choosable.

\section{Property~K and bipartite $(k,\ell)$-choosability}
\label{sec:propK}

Underlying the magnitude guarantee in Theorem~\ref{thm:bipartite} is the extremal study of another set theoretic property, one which is related to Property~B but which we have not found treated elsewhere in the literature.

For fixed integers $k,\ell\ge 2$, a family $\mathcal{F}\subseteq \binom{[\ell]}{k}$ has {\em Property~K$(k,\ell)$} if there exists a set $K \in \binom{[\ell]}{k-1}$ that intersects every set in $\mathcal{F}$. (The letter K stands for the Dutch word, {\em kleurrijk}.)
We then define $R(k,\ell)$ to be the cardinality of a smallest $\mathcal{F}\subseteq \binom{[\ell]}{k}$ that does not have Property~K$(k,\ell)$.
Clearly, $R(k,\ell)\ge k$ always.
Observe that $R(k,2k-1)=\binom{2k-1}{k}$ since the collection $\binom{[2k-1]}{k}$ of all $k$-subsets of $[2k-1]$ does not have Property~K$(k,2k-1)$, whereas any proper subcollection of $\binom{[2k-1]}{k}$ has Property~K$(k,2k-1)$.  It also holds that $M(k,\ell) = \infty$ if $\ell \le 2k-2$, as then every subcollection of $\binom{[\ell]}{k}$ trivially has Property~K$(k,\ell)$.

Let us now demonstrate the connection between Property~K$(k,\ell)$ and $(k,\ell)$-choosability of bipartite graphs with one part that is small enough.

\begin{proposition}\label{prop:bipartite,propK}
Suppose $G$ is a graph that admits a bipartition $V = A\cup B$ with $|A| < R(k,\ell)$. Then $G$ is $(k,\ell)$-choosable.
\end{proposition}

\begin{proof}
Let $L$ be any $(k,\ell)$-list-assignment of the graph $G$. We define an $L$-colouring $c$ of $G$ as follows. Since $|A| < R(k,\ell)$, the family $\{L(u): u\in A\}$ has Property~K$(k,\ell)$, i.e.~there is a set $K \in \binom{[\ell]}{k-1}$ such that $L(u)\cap K \ne \emptyset$ for all $u\in A$.  We set $c(u)$ to be an arbitrary colour of $L(u)\cap K$ for all $u\in A$.  Since $|K| = k - 1$ and the lists all have $k$ colours, we have for any $v \in B$ that $L(v)\setminus K \ne \emptyset$ and we set $c(v)$ to be an arbitrary colour of $L(v)\setminus K$. Clearly, $c$ defines a proper $L$-colouring of $G$, as required.
\end{proof}

As $R(k,2k-1)=\binom{2k-1}{k}$ and it is known that $\ch(K_{m,m^m})>m$ for any $m\ge 1$, we immediately obtain an affirmative answer to the second question of Kr\'al' and Sgall mentioned in the introduction.
Notice that $R(k,\ell) = k$ if $\ell \ge k^2$, by taking $\mathcal{F}$ to be an arbitrary partition of $[k^2]$ into $k$ $k$-subsets. 
On the other hand, borrowing classic arguments used to analyse Property~B~\cite{Erd63,Erd64}, we derive super-polynomial behaviour for $R(k,\ell)$ when $\ell = o(k^2/\ln k)$. 
More specifically, we have the following. 

\begin{theorem}\label{thm:propKext}
Let $k,\ell$ be integers such that $k\ge 2$ and $\ell\ge 2k-1$. Then
\begin{align*}
\frac{\ell!(\ell-2k+1)!}{(\ell-k)!(\ell-k+1)!} \le R(k,\ell) < \frac{\ell!(\ell-2k+1)!}{(\ell-k)!(\ell-k+1)!}\ln\binom{\ell}{k-1}.
\end{align*}
\end{theorem}

\noindent
Note that only the lower bound expression is needed for Theorem~\ref{thm:bipartite} and it is easily seen to be more than $\exp((k-1)^2/\ell)$.
As $k\to \infty$, one can check by Stirling's approximation that, if $\ell \sim b k$ for some fixed $b > 2$, then
\begin{align*}
R(k,\ell) = \left((b-2)^{b-2}(b-1)^{2-2b}b^b+o(1)\right)^k.
\end{align*}
Therefore, together with Proposition~\ref{prop:bipartite,propK} and the fact that $\ch(K_{m,m^m})>m$ for any $m\ge 1$, we conclude that Theorem~\ref{thm:bipartite} holds.

\begin{proof}[Proof of Theorem~\ref{thm:propKext}]
First we prove the lower bound.
Fix a family $\mathcal{F}\subseteq \binom{[\ell]}{k}$ with cardinality less than the leftmost expression. Choose $K \in \binom{[\ell]}{k-1}$ uniformly at random.
For any fixed $F \in \mathcal{F}$, we have
\begin{align*}
\Pr(F \cap K = \emptyset) 
&
 = \frac{\binom{\ell-k}{k-1}}{\binom{\ell}{k-1}}
 = \frac{(\ell-k)!(\ell-k+1)!}{\ell!(\ell-2k+1)!}. 
\end{align*}
By a union bound and the choice of cardinality of $\mathcal{F}$,
\begin{align*}
\Pr(F\cap K = \emptyset\text{ for some }F\in{\mathcal F}) \le
\sum_{F\in\mathcal{F}}\Pr(F \cap K = \emptyset) < 1.
\end{align*}
So with positive probability there is a set $K\in\binom{[\ell]}{k-1}$ certifying that $\mathcal{F}$ has Property~K$(k,\ell)$.

Next we prove the upper bound.
Fix $K \in \binom{[\ell]}{k-1}$. Let $F\in \binom{\ell}{k}$ be a set chosen uniformly at random. Then
\begin{align*}
\Pr(F \cap K = \emptyset) 
& = \frac{\binom{\ell-k+1}{k}}{\binom{\ell}{k}} = \frac{(\ell-k)!(\ell-k+1)!}{\ell!(\ell-2k+1)!}. 
\end{align*}
Let $\mathcal{F} = \{F_1,\dots,F_r\}$ be a family of sets chosen independently and uniformly at random from $\binom{[\ell]}{k}$. Based on the above calculation, we have that
\begin{align*}
\Pr( F_i \cap K \ne \emptyset\text{ for all }i\in\{1,\dots,r\}) \le \left(1-\frac{(\ell-k)!(\ell-k+1)!}{\ell!(\ell-2k+1)!}\right)^r.
\end{align*}
There are $\binom{\ell}{k-1}$ choices for $K$, so we have
\begin{align*}
&\Pr(\mathcal{F}\text{ has no $K$ certifying Property~K$(k,\ell)$}) \\
& \le \binom{\ell}{k-1}\exp\left(-r\frac{(\ell-k)!(\ell-k+1)!}{\ell!(\ell-2k+1)!}\right).
\end{align*}
This last expression is less than $1$ if
\begin{align*}
r > \frac{\ell!(\ell-2k+1)!}{(\ell-k)!(\ell-k+1)!}\ln\binom{\ell}{k-1},
\end{align*}
which establishes the upper bound.
\end{proof}

\begin{table}[!h]
\center
\begin{tabular}{|c||c|c|c|c|c|c|}
  \hline
  $\ell$ & $\leq 4$ & $5$ & $6$ & $7$ & $8$ & $\geq 9$ \\
  \hline
  $R(3,\ell)$ & $+\infty$ & $10$ & $8$ & $5$ & $4$ & $3$ \\
  \hline
\end{tabular}\caption{The complete table of values of $R(3,\ell)$.\label{table}}
\end{table}

\section{Conclusion}
\label{sec:conclusion}

In this work, we have considered the question of list colouring with a bounded palette and illustrated its connection to other parameters and tools in extremal combinatorics.
In particular, we answered the second of the two questions of Kr\'al' and Sgall described in the introduction, by showing that $C(k,2k-1)$ as defined in Theorem~\ref{thm:KrSg,positive} must be $\Omega(4^k/\sqrt{k})$. Moreover, using a connection to Property K, we showed that $C(k,\ell)$ must be super-polynomial in $k$ if $\ell = o(k^2/\ln k)$ and exponential if $\ell = O(k)$.
We also gave better upper bounds on $C(k,\ell)$ for large $k$ and $\ell \ge 2.75 k$ by a direct application of the recently-introduced containers method.

Except for the case $k=2$, the first question of Kr\'al' and Sgall is open.
Reiterating: what is the least $\ell^*$, if it exists, such that every graph is $(k+1)$-choosable if it is $(k,\ell^*)$-choosable?
The probabilistic methods used to prove Theorems~\ref{thm:KrSg,positive} and~\ref{thm:containers} appear too weak to prove the existence of such an $\ell^*$.
Moreover, Theorem~\ref{thm:bipartite} falls well short of refuting the existence of such an $\ell^*$.
To start with the smallest open case, is there an $\ell^*$ such that every $(3,\ell^*)$-choosable graph is $4$-choosable? If so, is the value suggested by Table~\ref{table} optimal, i.e.~is it true that every $(3,9)$-choosable graph is $4$-choosable?

\section*{Acknowledgement}

We thank the referees for their careful reading and helpful comments.

\bibliographystyle{abbrv}
\bibliography{palette}

\end{document}